%% file: Paper_Main2.tex
\documentclass{amsart}
\usepackage{amssymb,amsmath,bbm}
\usepackage{pdflscape}
\usepackage{mathrsfs}
\usepackage{tikz}
\usetikzlibrary{matrix,backgrounds}
\usepackage{graphicx}
\usepackage{caption}
\usepackage{subcaption}
\usepackage{afterpage}
\usepackage{hyperref}
\usepackage{rotating}
\usepackage{blkarray}

\usetikzlibrary{calc,intersections, arrows, shapes}

\newcommand{\CC}{\mathbb{C}}

\newcommand{\RR}{\mathbb{R}}

\newcommand{\ZZ}{\mathbb{Z}}
\newcommand{\NN}{\mathbb{N}}

\newcommand{\VV}{\mathcal{V}}

\newcommand{\rank}{\textup{rank}\,}

\newcommand{\conv}{\textup{conv}}

\newcommand{\xx}{\mathbf x}

\newcommand{\s}{\mathbf s}
\newcommand{\uu}{\mathbf u}
\newcommand{\vv}{\mathbf v}
\renewcommand{\tt}{\mathbf t}
\renewcommand{\aa}{\mathbf a}
\newcommand{\mcA}{\mathcal A}
\newcommand{\bb}{\mathbf b}

\newcommand{\mb}[1]{\boldsymbol #1}

\newtheorem{theorem}{Theorem}[section]

\newtheorem{lemma}[theorem]{Lemma}
\newtheorem{corollary}[theorem]{Corollary}
\newtheorem{proposition}[theorem]{Proposition}
\theoremstyle{definition}
\newtheorem{definition}[theorem]{Definition}
\newtheorem{example}[theorem]{Example}

\theoremstyle{remark}

\newtheorem{problem}[theorem]{Problem}

\title[Projectively unique polytopes and toric slack ideals]{Projectively unique polytopes\\ and toric slack ideals}

\author[Gouveia]{Jo{\~a}o Gouveia}
\address{CMUC, Department of Mathematics,
  University of Coimbra, 3001-454 Coimbra, Portugal}
\email{jgouveia@mat.uc.pt}

\author[Macchia]{Antonio Macchia}
\address{Discrete Geometry Group, Freie Universit\"at Berlin,
Arnimallee 2, 14195 Berlin, Germany}
\email{macchia.antonello@gmail.com}

\author[Thomas]{Rekha R. Thomas}
\address{Department of Mathematics, University of Washington, Box
  354350, Seattle, WA 98195, USA} \email{rrthomas@uw.edu}

\author[Wiebe]{Amy Wiebe}
\address{Department of Mathematics, University of Washington, Box
  354350, Seattle, WA 98195, USA} \email{awiebe@uw.edu}

\thanks{Gouveia was partially supported by FCT under grants UID/MAT/00324/2013 through CMUC, and P2020 SAICTPAC/0011/2015,
Macchia was supported by INdAM, Thomas by
the U.S. National Science Foundation grants DMS-1418728 and DMS-1719538, and Wiebe by NSERC}

\date{\today}

\begin{document}
\sloppy

\begin{abstract}
The slack ideal of a polytope is a saturated determinantal ideal that gives rise to a new model for the
realization space of the polytope. The simplest slack ideals are toric and have connections to projectively
unique polytopes. We prove that if a projectively unique polytope has a toric slack ideal, then
it is the toric ideal of the bipartite graph of vertex-facet non-incidences of the polytope. The slack ideal of a polytope is
contained in this toric ideal if and only if the polytope is morally $2$-level, a generalization of the $2$-level property
in polytopes. We show that polytopes that do not admit rational realizations cannot have toric slack ideals. A classical example of a projectively unique polytope with no rational realizations is due to Perles. We prove that the slack ideal of the Perles polytope is reducible, providing the first example of a slack ideal that is not prime.
\end{abstract}


\maketitle

\input{Intro2}

\input{Background2}

\input{Toric2}

\input{PU2}

\input{Perles}

\input{Conclusion2}

\bibliographystyle{alpha}
\bibliography{all}
\end{document}

%% file: Intro2.tex

%
%


\section{Introduction} \label{sec:introduction}

An important focus in the study of polytopes is the investigation of their realization spaces. Given a $d$-polytope $P \subset \RR^d$, its face lattice determines its combinatorial type. The realization space of $P$ is the set of all geometric realizations of polytopes in the combinatorial class of $P$.
A new model for the realization space of a polytope modulo projective transformations, called the {\em slack realization space}, was introduced in \cite{GMTWfirstpaper}. This model arises as the positive part of the
real variety of $I_P$, the {\em slack ideal} of $P$, which is a saturated determinantal ideal of a symbolic matrix whose zero pattern encodes the combinatorics of $P$. The slack ideal and slack realization space were extended to
matroids in \cite{BW18}.

The overarching goal of this paper is to initiate a study of the algebraic and geometric properties of slack ideals as they provide the main computational engine in our model of realization spaces. As shown in \cite{GMTWfirstpaper}, slack ideals can be used to answer many different questions about the realizability of polytopes. These ideals were introduced in \cite{GPRT17} where they were used to study the notion of {\em psd-minimality} of polytopes, a property of interest in optimization. Thus, developing the properties and understanding the implications of slack ideals can
directly impact both polytope and matroid theory.  Even as a purely theoretical object, slack ideals present a
new avenue for research in commutative algebra.

In this paper, we focus on the simplest possible slack ideals, namely, toric slack ideals.
Since slack ideals do not contain monomials, the simplest ones are generated by binomials. Toric ideals are precisely those binomial ideals that are prime. Toric slack ideals already form a rich class with important connections to projective uniqueness. In general, slack ideals offer a new classification scheme for polytopes via the algebraic properties and invariants of the ideal, and the toric case offers a nice example of this.
The vertex-facet (non)-incidence structure of a polytope~$P$ can be encoded in a bipartite graph whose toric ideal, $T_P$, plays a special role in this context. We call $T_P$ the {\em toric ideal of the {non-incidence} graph of} $P$,
and say that $I_P$ is {\em graphic} if it coincides with $T_P$. In Theorem~\ref{thm:graphic} we prove that $I_P$ is graphic if and only if $I_P$ is toric and $P$ is projectively unique. On the other hand,
there are infinitely many combinatorial types in high enough dimension that are projectively unique but do not have toric slack ideals, as well as non-projectively unique polytopes with toric slack ideals. We give several concrete examples.

The toric ideal $T_P$ has other interesting geometric connections. We prove that
$I_P$ is contained in $T_P$ if and only if $P$ is {\em morally 2-level},
which is a polarity-invariant property of a polytope that generalizes the notion of $2$-level polytopes \cite{StanleyCompressed}, \cite{BFFFMP17}, \cite{FFM16}, \cite{GS17}. Theorem~\ref{thm:morally 2-level} characterizes morally 2-level polytopes in terms of the slack variety. As a consequence we get that a polytope with no
rational realizations cannot have a toric slack ideal.

An important feature of a toric ideal is that the positive part of its real variety is Zariski dense in its complex variety. This implies that the toric ideal is the vanishing ideal of the positive part of its variety. In general, it is not easy to determine whether $I_P$ is the vanishing ideal of the positive part $\VV_+(I_P)$, of its variety $\VV(I_P)$.
We show that the slack ideal of a classical polytope due to Perles is reducible and that in this case, $\VV_+(I_P)$ is not Zariski dense in $\VV(I_P)$. This eight-dimensional polytope is projectively unique and does not have rational realizations.  It provides the first concrete instance of a slack ideal that is not prime.

{\bf Organization of the paper.}
In Section~\ref{sec:bg} we summarize the needed background on slack ideals of polytopes.
In Section~\ref{sec:Toric} we introduce $T_P$, the toric ideal of the non-incidence graph of a polytope $P$, and show its relationship to pure difference binomial slack ideals and morally $2$-level polytopes. We prove in Section~\ref{sec:PU} that slack ideals are graphic if and only if they are toric and the underlying polytope is projectively unique. In particular, we show that all $d$-polytopes with $d+2$ vertices or facets have graphic slack ideals, but this property holds beyond this class.  In this section we also illustrate toric slack ideals that do not come from projectively unique polytopes and the existence of projectively unique polytopes that do not have
toric slack ideals. We conclude in Section~\ref{sec:Perles} with the Perles polytope \cite[Section 5.5]{Grunbaum}. We show that the Perles polytope has a reducible slack ideal despite being projectively unique, providing the first concrete example of a non-prime slack ideal. In this case,
$\VV_+(I_P)$ is not Zariski dense in $\VV(I_P)$.

{\bf Acknowledgements.} We thank Arnau Padrol, David Speyer and G\"unter Ziegler for helpful conversations.
We also thank Marco Macchia\break for providing us with a list of known $2$-level polytopes, available at
\url{https://mmacchia90.bitbucket.io/data.html}, that helped us find interesting examples and counterexamples. We are indebted to the \texttt{SageMath} \cite{SageMath}, \texttt{Macaulay2} \cite{M2} and \texttt{Maple} \cite{Maple}  software systems for the computations in this paper.

%% file: Background2.tex

%
%

\section{Background: Slack Matrices and Ideals of Polytopes} \label{sec:bg}

We now give a brief introduction to slack matrices and slack ideals of polytopes.
For more details see \cite{slackmatrixpaper}, \cite{GPRT17} and \cite{GMTWfirstpaper}.

A $d$-dimensional polytope $P \subset \RR^d$ with $v$ {labelled}  vertices
and $f$ {labelled}  facet inequalities has two usual representations: a $\VV$-representation $P = \conv\{\mb{p}_1,\ldots, \mb{p}_v\}$ as the convex hull of vertices, and an $\mathcal{H}$-representation $P = \{\xx\in\RR^d : W\xx \leq \mb{w}\}$ as the intersection of the half spaces defined by the facet inequalities $W_j \xx \leq \mb{w}_j$, $j=1,\ldots, f$, where $W_j$ denotes the $j$th row of $W \in \RR^{f \times d}$.
Let $V \in \RR^{v \times d}$ be the matrix with rows ${\mb{p}_1}^\top,\ldots, {\mb{p}_v}^\top$,  and let
$\mathbbm{1} \in \RR^v$ be the vector of all ones.  The combined data of the two representations yields a {\em slack matrix} of $P$, defined as

\begin{equation} \label{EQ:slackdef} S_P := \left[\begin{array}{cc} \mathbbm{1} & V\\
 \end{array}\right] \left[\begin{array}{c} \mb{w}^\top \\ -W^\top \end{array}\right] \in\RR^{v\times f}. \end{equation}
 Since scaling the facet inequalities by positive real numbers does not change the polytope, $P$ in fact has
 infinitely many slack matrices of the form $S_P D_f$ where $D_f$ denotes a $f \times f$ diagonal
 matrix with positive entries on the diagonal. Also, affinely equivalent polytopes have the same set of slack
 matrices.

 Slack matrices were introduced in \cite{Yannakakis}.  The $(i,j)$-entry of $S_P$ is $\mb{w}_j - W_j\mb{p}_i$ which is the {\em slack} of the
 $i$th vertex $\mb{p}_i$ of $P$ with respect to the $j$th facet inequality $W_j^\top \xx \leq w_j$ of $P$. Since $P$ is a $d$-polytope, $\textup{rank}(\left[\begin{array}{cc} \mathbbm{1} & V\\
 \end{array}\right]) = d+1$, and hence, $\rank(S_P) = d+1$. Also, $\mathbbm{1}$ is in the column span of $S_P$. Further, the zeros in $S_P$ record the vertex-facet incidences of $P$, and hence the entire combinatorics (face lattice) of $P$. Interestingly, it follows from \cite[Theorem 22]{slackmatrixpaper} that any matrix with the above properties is in fact the slack matrix of a
 polytope that is combinatorially equivalent to $P$.

 \begin{theorem}
A nonnegative matrix $S \in \RR^{v \times f}$ is the slack matrix of a polytope in the combinatorial class of the labelled polytope $P$
if and only if the following hold:
\begin{enumerate}
\item $\textup{support}(S) = \textup{support}(S_P)$, \label{EQ:support}
\item $\rank(S) = \rank(S_P) = d+1$, and  \label{EQ:rank}
\item $\mathbbm{1}$ lies in the column span of $S$. \label{EQ:colspan}
\end{enumerate}
\label{THM:slackconditions}\end{theorem}

{This theorem gives rise to a new model for the realization space of $P$, as observed in \cite{GPRT17} and \cite{D14}.
  We briefly explain the construction of the slack model for the
realization space of $P$ from \cite{GPRT17}, developed further in \cite{GMTWfirstpaper}.
}

The {\em symbolic slack matrix}, $S_P(\xx)$, of $P$ is obtained by replacing the nonzero entry of $S_P$ in position
$(i,j)$ by the variable $x_{i,j}$. We assume that there are $t$ variables $x_{i,j}$ and let $\xx$ denote the collection of all $x_{i,j}$ (namely, those indexed by vertices $\mb{p}_i$ and facets $F_j$ with $\mb{p}_i\notin F_j$). 
The {\em slack ideal} of $P$ is the saturation of the ideal generated by the $(d+2)$-minors of $S_P(\xx)$, namely
\begin{equation} I_P := \langle (d+2)\text{-minors of }S_P(\xx) \rangle :\left(\prod_{(i,j): \mb{p}_i \notin F_j} x_{i,j}\right)^\infty \subset \CC[\xx].
\label{EQ:slackidealdef} \end{equation}
Note that since $I_P$ is saturated, it does not contain any monomials.
The {\em slack variety} of $P$ is the complex variety
$\mathcal{V}(I_P) \subset \CC^t$.
If $\s \in \CC^t$ is a zero of $I_P$, then we identify it with the matrix
$S_P(\s)$.

{By \cite[Corollary 1.5]{GPRT17}, two polytopes $P$ and $Q$ in the same combinatorial class are projectively equivalent if and only if $D_vS_PD_f$ is a slack matrix of $Q$ for some positive diagonal matrices $D_v,D_f$.
Using this fact and Theorem~\ref{THM:slackconditions}, we see that the positive part of $\VV(I_P)$, namely $\VV(I_P) \cap \RR^t_{>0} =: \VV_+(I_P)$,  leads to a realization space for $P$, modulo projective transformations.}

\begin{theorem} \cite{GMTWfirstpaper} Given a polytope $P$, there is a bijection
between the elements of $\VV_+(I_P)/(\RR^v_{>0}\times\RR^f_{>0})$ and the classes of projectively equivalent polytopes in the combinatorial class of $P$.
\end{theorem}

The space $\VV_+(I_P)/(\RR^v_{>0}\times\RR^f_{>0})$ is called the {\em slack realization space} of $P$.

%% file: Toric2.tex

%
%

\section{The toric ideal of the non-incidence graph of a polytope}
\label{sec:Toric}

We begin by defining the toric ideal $T_P$ of the non-incidence graph of a polytope~$P$.
In the next section we characterize when $T_P$ equals $I_P$ which relies on the projective uniqueness of $P$.
In this section we  examine the relationship between $I_P$ and $T_P$ and the
 implications of $I_P$ being contained in $T_P$.

First we recall the definition of a toric ideal. Let $\mcA=\{\aa_1, \dots, \aa_n\}$ be a point configuration in $\ZZ^d$. Sometimes we will identify $\mcA$ with the $d \times n$ matrix whose columns are the vectors $\aa_i$. Consider the $\CC$-algebra homomorphism
\[
\pi:\CC[x_1,\dots,x_n] \rightarrow \CC[t_1^{\pm 1},\dots,t_d^{\pm 1}], \text{ such that } x_j \mapsto \mathbf t^{\aa_j}.
\]
The kernel of $\pi$, denoted by $I_\mcA$, is called the \textit{toric ideal} of $\mcA$. The ideal $I_\mcA$ is binomial and prime (see \cite[Chapter 4]{Stu96}). More precisely, $I_\mcA$ is generated by homogeneous binomials:
\begin{equation}\label{eq:toricGens}
I_\mcA = \langle \xx^{\uu^+}-\xx^{\uu^-} \in \CC[x_1,\dots,x_n] : \uu \in \ker_{\ZZ}(\mcA)\rangle,
\end{equation}
where $\ker_{\ZZ}(\mcA)=\{\uu \in \ZZ^n : \mcA\uu=\mathbf 0\}$, $\uu=\uu^+ - \uu^-$, with $\uu^+,\uu^- \in \ZZ^n_{\geq 0}$ the positive and the negative parts of $\uu$.

Let $I_\mcA$ be a toric ideal and $V_\mcA=\mathcal V_\CC(I_\mcA)$ be its complex affine toric variety which is the Zariski closure of the set of points $\{(\tt^{\aa_1},\dots,\tt^{\aa_n}) : \tt \in (\CC^*)^d\}$. Define
\[
\phi_\mcA : (\CC^*)^d \rightarrow \CC^n, \qquad \tt \mapsto (\tt^{\aa_1},\dots,\tt^{\aa_n}),
\]
so that $V_\mcA=\overline{\phi_\mcA((\CC^*)^d)}$. We are interested in the positive part of $V_\mcA$, namely, $V_\mcA \cap \RR^n_{>0}$. Note that this set contains $\phi_\mcA(\RR^d_{>0})$.

The following result follows from the Zariski density of the positive part of a toric variety in its complex variety. However, we write an independent proof.

\begin{lemma}\label{lem:toricBinomials}
Let $I_\mcA$ be a toric ideal in $\CC[x_1,\dots,x_n]$. If $\uu,\vv \in \NN^n$ and $\xx^\uu-\xx^\vv$ vanishes on the set of points $\phi_\mcA(\RR^d_{>0})$, then $\xx^\uu-\xx^\vv \in I_\mcA$.
\end{lemma}

\begin{proof}
Notice that $\xx^\uu-\xx^\vv$ evaluated at any point $(\tt^{\aa_1},\dots,\tt^{\aa_n}) \in \phi_\mcA((\CC^*)^d)$ is just $\tt^{\mcA\uu}-\tt^{\mcA\vv}$. Then, since $\xx^\uu-\xx^\vv$ vanishes on $\phi_\mcA(\RR^d_{>0})$, we have that $\tt^{\mcA\uu}=\tt^{\mcA\vv}$ for all $\tt \in \RR^d_{>0}$. Thus, if we fix $i \in \{1,\dots,d\}$ and specialize to $t_j=1$ for all $j \neq i$, we get $t_i^{(\mcA\uu)_i}=t_i^{(\mcA\vv)_i}$ for all $t_i \in \RR_{>0}$, which means we must have $(\mcA\uu)_i=(\mcA\vv)_i$. Since this holds for all $i$, it follows that $\mcA\uu=\mcA\vv$, hence $\xx^\uu-\xx^\vv \in I_\mcA$ by \eqref{eq:toricGens}.
\end{proof}


\begin{definition}
Let $P$ be a $d$-polytope in $\RR^d$.
\begin{enumerate}
\item Define the {\em non-incidence graph} of $P$, denoted as $G_P$, to be the undirected bipartite graph on the vertices and
facets of $P$ with an edge connecting vertex $i$ to facet $j$ if and only if $i$ {\bf does not} lie on $j$.
\item Let $T_P$ be the toric ideal of $\mathcal A_P$, the vertex-edge incidence matrix of $G_P$. The matrix $\mathcal{A}_P$ has rows (columns)  indexed by the vertices (edges) of $G_P$, with $(i,j)$-entry equal to $1$ if vertex $i$ is incident to edge $j$ and $0$ otherwise. We call $T_P$ the
\textit{toric ideal of the non-incidence graph of $P$}.
\end{enumerate}
\end{definition}

Note that $G_P$ records the support of a slack matrix of $P$, and so we can think of its edges as being labelled by the corresponding entry of $S_P(\xx)$. Toric ideals of bipartite graphs have been studied in the literature.

\begin{lemma} [\textbf{\cite[Lemma 1.1]{OH99}, \cite[Theorem 10.1.5]{Vil15}}] \label{lem:toricGraphGens}
The ideal $T_P$ is generated by all binomials of the form $\xx^{C^+}-\xx^{C^-}$, where $C$ is an (even) chordless cycle in $G_P$, and
{$C^+, C^-\in\ZZ^{|E|}$ are the incidence vectors of the two sets of edges that partition $C$ into alternate edges (that is, if we orient edges from vertices to facets in $G_P$, then $C^+$ consists of the forward edges in a traversal of $C$, and $C^-$ the backward edges).}
Thus, for every even closed walk $W$ in $G_P$, and indeed any union of such,  $\xx^{W^+}-\xx^{W^-} \in T_P$.
\end{lemma}

\begin{example} \label{EX:toricIdeal}
Consider the $4$-polytope $P = \textup{conv}(0, 2e_1, 2e_2, 2e_3, e_1+e_2-e_3, e_4, e_3+e_4)$ \cite[Table 1. \#3]{GPRT17} where $e_i$ is the standard unit vector in $\RR^4$. This polytope is projectively unique with $f$-vector (7,17,17,7). It has symbolic slack matrix
\[
S_P(\xx) = \begin{blockarray}{cccccccc}
 & F_1 & F_2 & F_3 & F_4 & F_5 & F_6 & F_7 \\
 \begin{block}{c[ccccccc]}
\mb{p}_1 & 0 & x_{1,2} & 0 & 0 & 0 & x_{1,6} & 0\\
\mb{p}_2 & x_{2,1} & 0 & 0 & 0 & 0 & x_{2,6} & 0\\
\mb{p}_3 & x_{3,1} & 0 & x_{3,3} & 0 & 0 & 0 & x_{3,7}\\
\mb{p}_4 & 0 & x_{4,2} & x_{4,3} & 0 & 0 & 0 & x_{4,7}\\
\mb{p}_5 & 0 & 0 & 0 & 0 & x_{5,5} & 0 & x_{5,7}\\
\mb{p}_6 & 0 & 0 & 0 & x_{6,4} & x_{6,5} & x_{6,6} & 0\\
\mb{p}_7 & 0 & 0 & x_{7,3} & x_{7,4} & 0 & 0 & 0 \\
\end{block}
\end{blockarray}.
\]
Its non-incidence graph $G_P$ is given in Figure~\ref{fig:toricideal}. Notice that each edge of $G_P$ can be naturally labelled with the corresponding 
$x_{i,j}$ from $S_P(\xx)$. Under this labelling, the chordless cycle marked with dashed lines in Figure~\ref{fig:toricideal} corresponds to the binomial 
$x_{1,6}x_{2,1}x_{3,3}x_{4,2}-x_{1,2}x_{2,6}x_{3,1}x_{4,3}\in T_P$. One can check that the remaining generators of $T_P$, corresponding to chordless cycles of $G_P$, are
$$\begin{array}{ll}
x_{3,7}x_{4,3}-x_{3,3}x_{4,7}, &x_{4,7}x_{5,5}x_{6,4}x_{7,3}-x_{4,3}x_{5,7}x_{6,5}x_{7,4}, \\
     x_{3,7}x_{5,5}x_{6,4}x_{7,3}-x_{3,3}x_{5,7}x_{6,5}x_{7,4}, & x_{1,6}x_{4,2}x_{6,4}x_{7,3}-x_{1,2}x_{4,3}x_{6,6}x_{7,4},\\
     x_{2,6}x_{3,1}x_{6,4}x_{7,3}-x_{2,1}x_{3,3}x_{6,6}x_{7,4}, & x_{1,6}x_{4,2}x_{5,7}x_{6,5}-x_{1,2}x_{4,7}x_{5,5}x_{6,6},\\
     x_{2,6}x_{3,1}x_{5,7}x_{6,5}-x_{2,1}x_{3,7}x_{5,5}x_{6,6}, & x_{1,6}x_{2,1}x_{3,7}x_{4,2}-x_{1,2}x_{2,6}x_{3,1}x_{4,7}.
\end{array}$$

\begin{figure}
\begin{tikzpicture}[scale=0.4,line cap=round,line join=round,line width=.7pt]
{\tikzstyle{every node}=[circle,draw=black,fill=black,inner sep=0pt,minimum width=2.5pt]
\node[label=below:$\mb{p}_7$] (g) at (5.5,0) {};
\node[label=below:$\mb{p}_6$] (f) at (3.5,0) {};
\node[label=below:$\mb{p}_5$] (e) at (1.5,0) {};
\node[label=below:$\mb{p}_4$] (d) at (-0.5,0) {};
\node[label=below:$\mb{p}_3$] (c) at (-2.5,0) {};
\node[label=below:$\mb{p}_2$] (b) at (-4.5,0) {};
\node[label=below:$\mb{p}_1$] (a) at (-6.5,0) {};
\node[label=above:$F_7$] (G) at (5.5,6.5) {};
\node[label=above:$F_6$] (F) at (3.5,6.5) {};
\node[label=above:$F_5$] (E) at (1.5,6.5) {};
\node[label=above:$F_4$] (D) at (-0.5,6.5) {};
\node[label=above:$F_3$] (C) at (-2.5,6.5) {};
\node[label=above:$F_2$] (B) at (-4.5,6.5) {};
\node[label=above:$F_1$] (A) at (-6.5,6.5) {};
\draw[dashed] (a) -- (B);
\draw[dashed] (a) -- (F);
\draw[dashed] (b) -- (A);
\draw[dashed] (b) -- (F);
\draw[dashed] (c) -- (A);
\draw[dashed] (c) -- (C);
\draw (c) -- (G);
\draw[dashed] (d) -- (B);
\draw[dashed] (d) -- (C);
\draw (d) -- (G);
\draw (e) -- (E);
\draw (e) -- (G);
\draw (f) -- (D);
\draw (f) -- (E);
\draw (f) -- (F);
\draw (g) -- (C);
\draw (g) -- (D);
}
\end{tikzpicture}
\caption{Non-incidence graph $G_P$}
\label{fig:toricideal}
\end{figure}
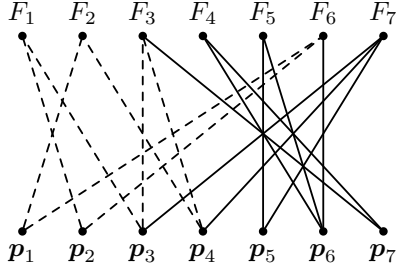
\end{example}

The toric ideal $T_P$ can coincide with $I_P$ as we will see in the next section.
For the remainder of this section we focus on the connections between $I_P$ and $T_P$.

\smallskip
An ideal is said to be a \textit{pure difference binomial ideal} if it is generated by binomials of the form $\xx^\aa - \xx^\bb$.
It follows from \eqref{eq:toricGens} that toric ideals are pure difference binomial ideals.  We now prove that if $I_P$ is toric, or more generally, a pure difference binomial ideal, then $I_P$ is always contained in $T_P$.

\begin{lemma}
If a binomial $\xx^\aa - \xx^\bb$ belongs to $I_P$, then it also belongs to $T_P$.
\end{lemma}

\begin{proof}
Let $p=\xx^\aa - \xx^\bb$. Each component $a_i$ of $\aa$ and $b_j$ of $\bb$ appears as the exponent of a variable in the symbolic slack matrix $S_P(\xx)$ and is hence indexed by an edge of $G_P$. Recall that all matrices obtained by scaling rows and columns of $S_P$ by positive scalars also lie in the real variety of $I_P$, and hence must vanish on $p$. This implies that the sum of the components of $\aa$ appearing as exponents of variables in a row (column) of $S_P(\xx)$ equals the sum of the components of $\bb$ appearing as exponents of variables in the same row (column).

Now think of the edges of $G_P$ in the support of $\aa$ as oriented from vertices of $P$ to facets of $P$ and edges in the support of $\bb$ as oriented in the opposite way. Then the previous statement is equivalent to saying that $p$ is supported on an oriented subgraph of $G_P$ (possibly with repeated edges) with the property that the in-degree and out-degree of every node in the subgraph are equal. Therefore, this subgraph is the vertex-disjoint union of closed walks in $G_P$, which by Lemma~\ref{lem:toricGraphGens} implies that~$p$ is in $T_P$.
\end{proof}

\begin{corollary} \label{cor:pure difference binomial ideals are in the toric ideal}
If $I_P$ is a pure difference binomial ideal, then $I_P \subseteq T_P$.
\end{corollary}

This containment can be strict as we see in the following example.

\begin{example} \label{ex:nongraphic}
Consider the $5$-polytope $P$ with vertices $\mb{p}_1,\ldots, \mb{p}_8$ given by
\begin{align*}
e_1,e_2,e_3,e_4,-e_1-2e_2-e_3,-2e_1-e_2-e_4,-2e_1-2e_2+e_5,-2e_1-2e_2-e_5
\end{align*}
where $e_1,\ldots, e_5$ are the standard basis vectors in $\RR^5$. It can be obtained by {splitting} the distinguished vertex $v$ of the vertex sum of two squares, $(\Box,v)\oplus(\Box,v)$ in the notation of \cite{McMull}. This polytope has 8 vertices and 12 facets and its symbolic slack matrix has the zero-pattern below
$$\begin{bmatrix}
0&*&0&0&0&0&*&0&0&0&0&0 \\
0&0&0&*&*&0&0&0&0&0&0&0 \\
0&0&0&0&0&*&*&*&0&0&*&* \\
*&0&0&*&0&*&0&0&*&0&*&0    \\
*&*&*&0&0&0&0&0&*&*&0&0    \\
0&0&*&0&*&0&0&*&0&*&0&*  \\
*&0&*&0&0&*&0&*&0&0&0&0 \\
0&0&0&0&0&0&0&0&*&*&*&*
\end{bmatrix}.$$

One can check using \texttt{Macaulay2} \cite{M2} that $I_P$ is toric and $I_P \subsetneq T_P$. In fact,
$\dim \CC[\xx]/I_P = 20$, while $\dim \CC[\xx]/T_P = 19$.
\end{example}

At first glance it might seem that if $I_P$ is contained in $T_P$ then
$I_P$ is a pure difference binomial ideal, but this is not true in general.

\begin{example} For the $3$-cube, $I_P \subsetneq T_P$. The toric ideal $T_P$ is minimally generated by
$80$ binomials, each corresponding to a chordless cycle in $G_P$,  while $I_P$ is minimally generated by
$222$ polynomials many of which are not binomials.
\end{example}

In fact, one can attach a geometric meaning to polytopes for which $I_P \subseteq T_P$.
A polytope $P$ is said to be \textit{$2$-level} if it has a slack matrix in which every positive entry is one, i.e.,
$S_P(\mathbbm{1})$ is a slack matrix of $P$.  This class of polytopes have received a great deal of attention
in the literature \cite{StanleyCompressed}, \cite{BFFFMP17}, \cite{FFM16}, \cite{GS17} and are also known as
{\em compressed polytopes}.

\begin{definition} \label{def:morally 2-level}
We call a polytope $P$ \textit{morally $2$-level} if $S_P(\mathbbm{1})$ lies in the slack variety of $P$.
\end{definition}


Note that if $P$ is morally $2$-level, it might not be that $S_P(\mathbbm{1})$ is a slack matrix of~$P$, but merely that
$\mathbbm{1} \in \VV_+(I_P)$. Hence, morally $2$-level polytopes contain $2$-level polytopes.
These polytopes correspond to pointed polyhedral cones having a choice of generators such that there is a 0/1 slack matrix of that cone.
For example, all regular $d$-cubes are $2$-level and hence any polytope that is combinatorially a $d$-cube is morally $2$-level but not necessarily $2$-level. Being morally $2$-level
does not require that there is a polytope in the combinatorial class of $P$ that is a $2$-level polytope. For example, a bisimplex in $\RR^3$ is morally $2$-level, but no polytope in its combinatorial class is $2$-level.
This is since $S_P(\mathbbm{1})$ can lie in the slack variety of $P$ even though it may not have the all-ones vector in its column space. A very attractive feature of the set of morally $2$-level polytopes is that it is closed under polarity
unlike the set of $2$-level polytopes, but preserves many of the properties of $2$-level polytopes such as psd-minimality \cite{GRT}, \cite{GPRT17}.

\begin{theorem} \label{thm:morally 2-level}
A polytope $P$ is morally $2$-level if and only if $I_P \subseteq T_P$.
\end{theorem}

\begin{proof}
Notice that the ideal $J_P = \langle (d+2)\text{-minors of }S_P(\xx) \rangle$ is contained in the slack ideal $I_P$.
Suppose that $S_P(\mathbbm{1})\in \mathcal{V}(I_P)$. Then any $(d+2)$-minor $p$ of $S_P(\xx)$ must have the same number of monomials with coefficient $+1$ as those with coefficient $-1$ since $p$ must vanish on $S_P(\mathbbm{1})$, which sets each monomial to one. This implies that we can write $p$ as a sum of pure difference binomials. Since $p$ is a minor, each of these pure difference binomials corresponds to a pair of permutations that induce two perfect matchings on the same set of vertices.
The union of these two matchings is a subgraph of $G_P$, which we can view as a directed graph by orienting the two matchings in opposite directions. Then each vertex will have equal in-degree and out-degree, which shows that these edges form a union of closed walks in $G_P$, and
thus the corresponding binomial is in $T_P$ by Lemma~\ref{lem:toricGraphGens}.
Therefore $p \in T_P$, so that $J_P\subseteq T_P$. Since toric ideals are saturated with respect to all variables, the result follows.

Conversely, suppose $I_P\subseteq T_P$. Since $T_P$ is generated by pure difference binomials, which vanish when evaluated at $S_P(\mathbbm{1})$, we have $S_P(\mathbbm{1}) \in \mathcal{V}(T_P)$. But $I_P\subseteq T_P$ implies that $\mathcal{V}(I_P)\supseteq \mathcal{V}(T_P) \ni S_P(\mathbbm{1})$, which is the desired result.
\end{proof}

We have talked about pure difference binomial slack ideals as a superset of toric slack ideals.
A slack ideal is binomial if it is generated by binomials of the form $\xx^\aa - \gamma \xx^\bb$, where $\gamma$ is a non-zero scalar.
Therefore, one might extend the study of toric slack ideals to the following hierarchy of binomial slack ideals:
\[
\textup{toric}  \subseteq \textup{pure difference binomial}  \subseteq \textup{binomial}.
\]
So far, we have not encountered a pure difference binomial slack ideal that is not toric, nor a binomial slack ideal which is not pure difference, but it might be possible that all containments are strict. It follows from Corollaries 2.2 and 2.5 in \cite{ES96} that, if the slack ideal $I_P$ is binomial, then it is a radical lattice ideal. This implies that the slack variety is a union of scaled toric varieties.

%% file: PU2.tex

%
%

\section{Projective uniqueness and toric slack ideals}
\label{sec:PU}

Recall that a polytope $P$ is said to be \textit{projectively unique} if any polytope $Q$ that is combinatorially equivalent to $P$ is also projectively equivalent to $P$, i.e., there is a projective transformation that sends $Q$ to $P$. This corresponds to saying that
the slack realization space of $P$ is a single positive point.

Every $d$-polytope with $d+2$ vertices or facets is projectively unique \cite[Exercise 4.8.30 (i)]{Grunbaum}.
In particular, all products of simplices are projectively unique. We first prove that the slack ideal of a $d$-polytope with $d+2$ vertices or facets coincides with $T_P$, and is thus toric.

\begin{proposition} \label{Prop:d+2vertices}
Let $P$ be a polytope in $\RR^d$ with $d+2$ vertices or facets. Then its slack ideal $I_P$ equals the toric ideal $T_P$.
\end{proposition}

\begin{proof}
Up to polarity we may consider $P$ to be a polytope with $d+2$ vertices. In this case $P$ is combinatorially equivalent to a repeated pyramid over a free sum of two simplices, $\mathrm{pyr}_r(\Delta_k \oplus \Delta_\ell)$, with $k,\ell \geq 1$, $r \geq 0$ and $r+k+\ell=d$ \cite[Section 6.1]{Grunbaum}. Since taking pyramids preserves the slack ideal, it is enough to study the slack ideals of free sums of simplices (respectively, product of simplices). By \cite[Lemma 5.7]{GPRT17}, if $P=\Delta_k \oplus \Delta_\ell$, then $S_P(\xx)$ has the zero pattern of the vertex-edge incidence matrix of the complete bipartite graph $K_{k+1,\ell+1}$.

From \cite[Proposition 5.9]{GPRT17}, it follows that $I_P$ is generated by the binomials
\[
\det(M_C) = \left|\begin{array}{cccccc}
x_1 & 0 & 0 & \cdots & 0 & x_2 \\
x_3 & x_4 & 0 & \cdots & 0 & 0 \\
0 & x_5 & x_6 & \cdots & 0 & 0 \\
\vdots & \vdots & \vdots & \ddots & \vdots & \vdots \\
0 & 0 & 0 & \cdots & x_{2c-2} & 0 \\
0 & 0 & 0 & \cdots & x_{2c-1} & x_{2c}
\end{array}\right|,
\]
where $M_C$ is a $c \times c$ symbolic matrix whose support is the vertex-edge incidence matrix of the simple cycle $C$ (of size $c$) in $K_{k+1,\ell+1}$.

On the other hand, $T_P$ is generated by the binomials $\xx^{D^+}-\xx^{D^-}$ corresponding to chordless cycles $D$ of the non-incidence graph $G_P$ by Lemma \ref{lem:toricGraphGens}. Thus, it suffices to show that there exists a bijection between simple cycles $C$ in $K_{k+1,l+1}$ and chordless cycles $D$ in $G_P$ such that $\det(M_C)=\xx^{D^+}-\xx^{D^-}$.

Let $v_1,\ldots, v_{k+\ell+2}$ be the vertices of $P$ and $F_1,\ldots, F_{(k+1)(\ell+1)}$ be its facets.
Since $S_P(\xx)$ has the support of the vertex-edge incidence matrix of $K_{k+1,\ell+1}$, we can consider $K_{k+1,\ell+1}$ to be a bipartite graph on the vertices $v_1,\dots,v_{k+\ell+2}$ where each edge $\{v_{i_1}, v_{i_2}\}$ corresponds exactly to the facet $F_j$ of $P$ containing neither $v_{i_1}$ nor $v_{i_2}$. Notice that the non-incidence graph $G_P$ can be obtained by subdividing each edge $\{v_{i_1}, v_{i_2}\}$ of $K_{k+1,\ell+1}$ into two edges $\{v_{i_1}, F_{j}\}$ and $\{F_{j}, v_{i_2}\}$.

Now, let $C$ be a simple cycle of size $c$ in $K_{k+1,\ell+1}$ with vertices $v_{i_1}, v_{i_2}, \dots, v_{i_c}$ and assume that $F_{j_1},F_{j_2},\dots,F_{j_c}$ are the facets corresponding to the edges of $C$. Then in $G_P$ there is a cycle $D$ of size $2c$ on vertices $v_{i_1},F_{j_1}, v_{i_2}, F_{j_2}, \dots, F_{j_{c-1}}, v_{i_c}, F_{j_c}$. In fact, one can see that the subgraph induced by these vertices is exactly a chordless cycle in $G_P$. This is because from the support of $S_P$ we know each facet in $P$ corresponds to a vertex of degree $2$ in $G_P$; furthermore, every edge in $G_P$ must be between a vertex and a facet, but since every facet already has degree $2$ in the cycle~$D$, this subgraph must consist only of this cycle.
Hence from a simple cycle $C$ in $K_{k+1,\ell+1}$, we get a chordless cycle $D$ in $G_P$, as desired. The reverse correspondence is analogous.
\end{proof}

The class of polytopes for which $I_P = T_P$ is larger than those with $d+2$ vertices or facets.

\begin{example} \label{ex:class3}
For the polytope given in Example~\ref{EX:toricIdeal}, which was $4$-dimensional but with $7$ vertices and $7$ facets, one can check that $I_P$ is the toric ideal $T_P$.
\end{example}

In $\RR^2$ the only projectively unique polytopes are triangles and squares. In $\RR^3$ there are four combinatorial classes of projectively unique polytopes --- tetrahedra, square pyramids, triangular prisms and bisimplices. The number of projectively unique $4$-polytopes is currently unknown. There are $11$ known combinatorial classes,
attributed to Shephard by McMullen \cite{McMull}, and listed in full in \cite{AZ15}. Beyond the $4$-polytopes with $4+2 = 6$ vertices or facets, this list has three additional combinatorial classes.  One of them is the polytope seen in Example~\ref{ex:class3}. It was shown in \cite{GPRT17} that all of the $11$ known projectively unique polytopes in $\RR^4$ have toric slack ideals. This discussion suggests that there might be a connection between projective uniqueness of a polytope and its slack ideal being toric. In this section we establish the precise result. The toric ideal $T_P$ of the non-incidence graph $G_P$ will again play an important role.

\begin{definition}
We say that the slack ideal $I_P$ of a polytope $P$ is \textit{graphic} if it is equal to the toric ideal $T_P$.
\end{definition}

\begin{theorem} \label{thm:graphic}
The slack ideal $I_P$ of a polytope $P$ is graphic if and only if $P$ is projectively unique and $I_P$ is toric.
\end{theorem}

\begin{proof}
Suppose that $I_P$ is graphic. Then, $I_P$ is toric, so we only need to show that~$P$ is projectively unique. Pick a maximal spanning forest $F$ of the bipartite graph $G_P$. By Lemma \ref{lem:forest} we may scale the rows and columns of $S_P$ so that it has ones in the entries indexed by $F$. Take an edge of $G_P$ outside of $F$ and consider the binomial corresponding to the unique cycle this edge forms together with $F$. Since $I_P = T_P$, this binomial is in $I_P$, therefore it must vanish on the above scaled slack matrix of $P$.  This implies that the entry in the slack matrix indexed by the chosen edge must also be $1$. Repeating this argument we see that the
entire slack matrix has $1$ in every non-zero entry which implies that there is only one possible slack matrix for $P$ up to scalings, hence only one polytope in the combinatorial class of~$P$ up to projective equivalence.

Conversely, suppose that $P$ is projectively unique and $I_P$ is toric, say $I_P = I_\mathcal{A}$ for some point configuration $\mathcal{A}$. Let $\xx^{\uu}-\xx^{\vv}$ be a generator of $T_P$.  Notice this generator vanishes when each $x_i = 1$, and
by Lemma~\ref{lem:toricGraphGens}, $\xx^{\uu}-\xx^{\vv} = \xx^{C^+}-\xx^{C^-}$ for some chordless cycle $C$ of $G_P$.
Now, since $I_P$ is toric, by Corollary~\ref{cor:pure difference binomial ideals are in the toric ideal} we have that $I_P \subseteq T_P$, and then by Theorem~\ref{thm:morally 2-level}, $S_P(\mathbbm 1) \in \mathcal V(I_P)$. Since $P$ is projectively unique, every element of $\mathcal{V}_+(I_P)$ is obtained by positive row and column scalings of $S_P(\mathbbm 1)$.
Therefore, $\phi_{\mathcal{A}}(\mathbb R^d_{>0})
\subseteq \mathcal{V}_+(I_P)$ consists of row and column scalings of $S_P(\mathbbm 1)$.
Since a binomial of the form $\xx^{C^+}-\xx^{C^-}$, where $C$ is a chordless cycle, contains in each of its monomials exactly one variable from each row and column of $S_P(\xx)$ on which it is supported, it must also vanish on all row and column scalings of $S_P(\mathbbm{1})$.
It follows that the generator $\xx^{\uu}-\xx^{\vv}$ vanishes on $\phi_{\mathcal A}(\mathbb R^d_{>0})$. By Lemma~\ref{lem:toricBinomials}, this means that $\xx^{\uu}-\xx^{\vv} \in I_P$, thus all generators of $T_P$ are contained in $I_P$, which completes the proof.
\end{proof}

Theorem~\ref{thm:graphic} naturally leads to the question whether $P$ can have a toric slack ideal even if it is not projectively unique and whether all projectively unique polytopes have toric slack ideals. In the rest of this section, we discuss these two questions.

All $d$-polytopes with toric slack ideals for $d\leq 4$ were found in \cite{GPRT17}. These polytopes all happen to be projectively unique, and hence have graphic slack ideals.
Therefore the first possible non-graphic toric slack ideal has to come from a polytope of dimension at least five. {Indeed, we saw that the polytope in Example~\ref{ex:nongraphic} has a toric slack ideal but is not graphic. Hence, this polytope is not projectively unique by Theorem~\ref{thm:graphic}, recovering a result implied by a theorem of McMullen \cite[Theorem 5.3]{McMull}.}

In the next section we will see a concrete $8$-polytope that is projectively unique but does not have a toric slack ideal. However, this is not an isolated instance as there are infinitely many such examples in high enough dimension.

\begin{proposition} \label{prop:many PU}
For $d \geq 69$ there exist infinitely many projectively unique $d$-polytopes that do not have a toric (even pure difference binomial) slack ideal.
\end{proposition}

\begin{proof}
In \cite{AZ15}, Adiprasito and Ziegler have shown that for $d \geq 69$ there are infinitely many projectively unique $d$-polytopes.
On the other hand, it follows from results in \cite{GPRT17} concerning semidefinite lifts of polytopes that in any dimension, there can only be finitely many combinatorial classes of polytopes whose slack ideal is a pure difference binomial ideal. 
\end{proof}

%% file: Perles.tex

%
%

\section{The Perles polytope has a reducible slack ideal}
\label{sec:Perles}

We now consider a classical example of a projectively unique polytope with no rational realization due to Perles \cite[p.94]{Grunbaum}. This is an $8$-polytope with $12$ vertices and $34$ facets with the additional feature that it has a non-projectively unique face. It is minimal in the sense that every $d$-polytope with at most $d+3$ vertices is rationally realizable.
We will show that the Perles polytope does not have a toric slack ideal and that in fact, its slack ideal is not prime, providing the first such example.

The non-existence of rational realizations of a polytope immediately implies that its slack ideal is not toric.
This is a corollary of Theorem \ref{thm:morally 2-level}.

\begin{corollary} \label{cor:non-rational not toric}
Let $P$ be a polytope in $\RR^d$ with no rational realization. Then $I_P$  cannot be a pure difference binomial ideal and, in particular, cannot be toric.
\end{corollary}

\begin{proof}
If $P$ has no rational realization, then $S_P(\mathbbm{1})$ does not lie in the slack variety of~$P$, since a rational point in $\VV_+(I_P)$ yields a rational realization of $P$ by \cite[Lemma 4.1]{GMTWfirstpaper}.
Therefore, by Theorem~\ref{thm:morally 2-level}, $I_P$ is not contained in $T_P$. Now applying Corollary~\ref{cor:pure difference binomial ideals are in the toric ideal}, we can conclude that $I_P$ is not a pure difference binomial ideal and, in particular, is not toric.
\end{proof}

The Perles polytope $P$ is constructed in \cite[p.95]{Grunbaum} from its affine Gale diagram shown in
Figure~\ref{fig:pointconfig}. This planar configuration stands in for the vector configuration in $\RR^3$ (Gale diagram) consisting of
$12$ vectors --- the eight vectors $A,B,C,D,E,F,G,H$ indicated with black dots that have $x_3=1$ and the four vectors
$-F,-G,-H,-I$ indicated with open circles that have $x_3=-1$. This means that $P$ has $12$ vertices and is of dimension $12-3-1=8$. The facets of $P$ are in bijection with the $34$ minimal positive circuits of the Gale diagram. Computing these, we get the support of the slack matrix $S_P$ shown below.

\begin{figure}[ht!]
\begin{tikzpicture}[scale=0.5,line cap=round,line join=round,line width=.7pt]
\clip(-4.1,-1) rectangle (6.2,5.8);
{\tikzstyle{every node}=[circle,draw=black,fill=black,inner sep=0pt,minimum width=2.5pt]
\draw (0.,0.)-- (2.,0.);
\draw (2.,0.)-- (2.618033988749895,1.9021130325903064);
\draw (2.618033988749895,1.9021130325903064)-- (1.,3.077683537175253);
\draw (1.,3.077683537175253)-- (-0.6180339887498947,1.9021130325903073);
\draw (-0.6180339887498947,1.9021130325903073)-- (0.,0.);
\draw (-1.6180339887498945,4.979796569765561)-- (2.,0.);
\draw (3.6180339887498953,4.9797965697655595)-- (0.,0.);
\draw (5.236067977499787,0.)-- (-0.6180339887498947,1.9021130325903073);
\draw (-3.2360679774997907,0.)-- (2.618033988749895,1.9021130325903064);
\draw (-1.6180339887498945,4.979796569765561)-- (-0.6180339887498947,1.9021130325903073);
\draw (-1.6180339887498945,4.979796569765561)-- (1.,3.077683537175253);
\draw (1.,3.077683537175253)-- (3.6180339887498953,4.9797965697655595);
\draw (3.6180339887498953,4.9797965697655595)-- (2.618033988749895,1.9021130325903064);
\draw (2.,0.)-- (5.236067977499787,0.);
\draw (2.618033988749895,1.9021130325903064)-- (5.236067977499787,0.);
\draw (-3.2360679774997907,0.)-- (0.,0.);
\draw (-3.2360679774997907,0.)-- (-0.6180339887498947,1.9021130325903073);
\node [label=below:$E$] at (0.,0.) {};
\node [draw=black, fill=white, minimum width=5pt] at (2.,0.) {};
\node [label=below:$F$] at (2.,0.) {};
\node [draw=black, fill=white, minimum width=5pt] at (2.618033988749895,1.9021130325903064) {};
\node [label={[label distance=1mm]10:$H$}] at (2.618033988749895,1.9021130325903064) {};
\node [draw=black, fill=white, minimum width=5pt] at (-0.6180339887498947,1.9021130325903073) {};
\node [label={[label distance=1mm]170:$G$}] at (-0.6180339887498947,1.9021130325903073) {};
\node [label=above:$C$] at (-1.6180339887498945,4.979796569765561) {};
\node [label=above:$D$] at (3.6180339887498953,4.9797965697655595) {};
\node [label=below:$B$] at (5.236067977499787,0.) {};
\node [label=below:$A$] at (-3.2360679774997907,0.) {};
\node [label=above:$I$, draw=black, fill=white] at (1,1.38) {};}
\end{tikzpicture}
\vspace{-2mm}
\caption{The Perles Gale diagram.}
\label{fig:pointconfig}
\end{figure}

\[ \footnotesize{
\left[
{\arraycolsep=2.4pt\def\arraystretch{1.2}
\begin{array}{cccccccccccccccccccccccccccccccccc}
0 & 0 & 0 & \ast & \ast & \ast & 0 & 0 & 0 & 0 & 0 & 0 & 0 & \ast & \ast & \ast & \ast & \ast & \ast & \ast & \ast & 0 & 0 & 0 & 0 & 0 & 0 & 0 & 0 & 0 & 0 & 0 & 0 & 0\\
0 & 0 & 0 & \ast & 0 & 0 & \ast & \ast & \ast & 0 & 0 & 0 & 0 & \ast & \ast & \ast & \ast & 0 & 0 & 0 & 0 & \ast & \ast & \ast & \ast & 0 & 0 & 0 & 0 & 0 & 0 & 0 & 0 & 0\\
0 & 0 & 0 & 0 & 0 & 0 & \ast & 0 & 0 & \ast & \ast & 0 & 0 & \ast & \ast & 0 & 0 & \ast & \ast & 0 & 0 & \ast & 0 & 0 & 0 & \ast & \ast & 0 & 0 & 0 & 0 & 0 & 0 & 0\\
0 & 0 & 0 & 0 & \ast & 0 & 0 & 0 & 0 & 0 & 0 & \ast & \ast & 0 & 0 & \ast & \ast & 0 & 0 & \ast & 0 & 0 & \ast & \ast & 0 & 0 & 0 & \ast & \ast & \ast & 0 & 0 & 0 & 0\\
0 & 0 & 0 & 0 & 0 & 0 & 0 & \ast & 0 & \ast & 0 & \ast & 0 & 0 & 0 & 0 & 0 & 0 & 0 & 0 & 0 & \ast & \ast & 0 & 0 & \ast & 0 & \ast & \ast & 0 & \ast & \ast & \ast & 0\\
\ast & 0 & 0 & 0 & 0 & 0 & 0 & 0 & 0 & 0 & \ast & 0 & \ast & 0 & 0 & 0 & 0 & \ast & 0 & \ast & 0 & 0 & 0 & 0 & 0 & 0 & \ast & 0 & 0 & \ast & 0 & 0 & 0 & \ast\\
0 & \ast & 0 & 0 & 0 & 0 & 0 & 0 & \ast & 0 & 0 & 0 & 0 & 0 & 0 & 0 & 0 & 0 & 0 & 0 & 0 & 0 & 0 & \ast & \ast & 0 & 0 & 0 & 0 & \ast & \ast & 0 & 0 & \ast\\
0 & 0 & \ast & 0 & 0 & \ast & 0 & 0 & 0 & 0 & 0 & 0 & 0 & 0 & 0 & 0 & 0 & 0 & \ast & 0 & \ast & 0 & 0 & 0 & 0 & \ast & 0 & 0 & 0 & 0 & \ast & \ast & \ast & \ast\\
\ast & 0 & 0 & \ast & 0 & 0 & 0 & \ast & 0 & 0 & 0 & 0 & 0 & 0 & 0 & 0 & 0 & 0 & 0 & 0 & \ast & 0 & 0 & 0 & \ast & 0 & 0 & \ast & 0 & 0 & 0 & \ast & \ast & 0\\
0 & \ast & 0 & 0 & \ast & 0 & 0 & 0 & 0 & \ast & 0 & 0 & 0 & \ast & 0 & 0 & 0 & \ast & \ast & 0 & \ast & 0 & 0 & 0 & 0 & 0 & \ast & \ast & \ast & 0 & 0 & \ast & 0 & 0\\
0 & 0 & \ast & 0 & 0 & 0 & \ast & 0 & 0 & 0 & 0 & 0 & \ast & 0 & 0 & \ast & 0 & 0 & 0 & 0 & 0 & 0 & \ast & \ast & \ast & 0 & \ast & 0 & \ast & 0 & 0 & 0 & 0 & 0\\
0 & 0 & 0 & 0 & 0 & \ast & 0 & 0 & \ast & 0 & \ast & \ast & 0 & 0 & \ast & 0 & \ast & 0 & 0 & \ast & 0 & \ast & 0 & 0 & 0 & \ast & 0 & 0 & 0 & \ast & \ast & 0 & \ast & \ast
\end{array}}
\right]}
\]
\vspace{1mm}

It is straightforward to obtain $S_P(\xx)$ from the above matrix, but a direct calculation of the slack ideal of this example is challenging. Therefore, we resort to a
scaling technique that makes slack ideal computations easier.  The idea is to work with a
subvariety of the slack variety that contains a representative for every orbit under row and
column scalings. We do this by fixing as many entries as possible in $S_P(\xx)$ to one.
Having less variables, the slack ideal becomes easier to compute.
The non-incidence graph $G_P$ from Section 3 provides a systematic way to scale a maximal number of
entries in $S_P(\xx)$ to one.

\begin{lemma} \label{lem:forest}
Given a polytope $P$, we may scale the rows and columns of its slack matrix so that it has ones in the entries indexed by the edges in a maximal spanning forest $F$ of the graph $G_P$.
\end{lemma}

\begin{proof}
For every tree $T$ in the forest, pick a vertex to be its root, and orient the edges away from it. Now for each tree, pick the edges leaving the root and set to one the corresponding entry of $S_P$ by scaling the row or column corresponding to the destination vertex of the edge. Continue the process with the edges leaving the vertices just used and so on, until the trees are exhausted. Notice that once we fix an entry, the only way for us to change it again is by scaling either its row or column, which would mean in the graph that we would revisit one of the nodes of its corresponding edge. But this would imply the existence of a cycle in $F$, so by the time this process ends we have precisely the intended variables set to one.
\end{proof}

Even after the above scaling trick, the symbolic slack matrix of the Perles polytope has $75$ variables
which is challenging to work with. Therefore, we will work with a subideal of $I_P$.

Consider the following submatrix of $S_P(\xx)$ coming from its first 13 columns.
$$\begin{bmatrix}
0 & 0 & 0 & x_{1} & x_{2} & x_{3} & 0 & 0 & 0 & 0 & 0 & 0 & 0\\
0 & 0 & 0 & x_4 & 0 & 0 & x_5 & x_6 & x_7 & 0 & 0 & 0 & 0\\
0 & 0 & 0 & 0 & 0 & 0 & x_8 & 0 & 0 & x_9 & x_{10} & 0 & 0\\
0 & 0 & 0 & 0 & x_{11} & 0 & 0 & 0 & 0 & 0 & 0 & x_{12} & x_{13}\\
0 & 0 & 0 & 0 & 0 & 0 & 0 & x_{14} & 0 & x_{15} & 0 & x_{16} & 0\\
x_{17} & 0 & 0 & 0 & 0 & 0 & 0 & 0 & 0 & 0 & x_{18} & 0 & x_{19}\\
0 & x_{20} & 0 & 0 & 0 & 0 & 0 & 0 & x_{21} & 0 & 0 & 0 & 0\\
0 & 0 & x_{22} & 0 & 0 & x_{23} & 0 & 0 & 0 & 0 & 0 & 0 & 0\\
x_{24} & 0 & 0 & x_{25} & 0 & 0 & 0 & x_{26} & 0 & 0 & 0 & 0 & 0\\
0 & x_{27} & 0 & 0 & x_{28} & 0 & 0 & 0 & 0 & x_{29} & 0 & 0 & 0\\
0 & 0 & x_{30} & 0 & 0 & 0 & x_{31} & 0 & 0 & 0 & 0 & 0 & x_{32}\\
0 & 0 & 0 & 0 & 0 & x_{33} & 0 & 0 & x_{34} & 0 & x_{35} & x_{36} & 0
\end{bmatrix}.$$

The ideal of $10 \times 10$ minors of this submatrix, saturated by all its variables is clearly a subideal of $I_P$.
Using the scaling lemma we first set $x_i=1$ for $i = 1,4,5,6,7,8,9,10,13,15,16,17,18,21,22,26,27,28,29,30,31,32,33,35$. The resulting scaled slack subideal is:
\begin{gather*}
\langle \boldsymbol{x_{36}^2+x_{36}-1}, x_{34}-x_{36}-1, x_{25}-x_{36}, x_{24}-x_{36}, x_{23}-1,x_{20}-x_{36},\\
x_{19}-x_{36},x_{14}-x_{36}-1, x_{12}-x_{36},x_{11}-1,x_{3}-1,x_{2}-x_{36}-1 \rangle.
\end{gather*}
This means that
after scaling, the first 13 columns of every matrix $S_P(\mathbf{s})$ obtained from $\mathbf{s} \in \mathcal{V}(I_P)$ with
full support must have the form
\begin{equation} \begin{bmatrix}
0 & 0 & 0 & 1 & \alpha+1 & 1 & 0 & 0 & 0 & 0 & 0 & 0 & 0\\
0 & 0 & 0 & 1 & 0 & 0 & 1 & 1 & 1 & 0 & 0 & 0 & 0\\
0 & 0 & 0 & 0 & 0 & 0 & 1 & 0 & 0 & 1 & 1 & 0 & 0\\
0 & 0 & 0 & 0 & 1 & 0 & 0 & 0 & 0 & 0 & 0 & \alpha & 1\\
0 & 0 & 0 & 0 & 0 & 0 & 0 & \alpha+1 & 0 & 1 & 0 & 1 & 0\\
1 & 0 & 0 & 0 & 0 & 0 & 0 & 0 & 0 & 0 & 1 & 0 & \alpha\\
0 & \alpha & 0 & 0 & 0 & 0 & 0 & 0 & 1 & 0 & 0 & 0 & 0\\
0 & 0 & 1 & 0 & 0 & 1 & 0 & 0 & 0 & 0 & 0 & 0 & 0\\
\alpha & 0 & 0 & \alpha & 0 & 0 & 0 & 1 & 0 & 0 & 0 & 0 & 0\\
0 & 1 & 0 & 0 & 1 & 0 & 0 & 0 & 0 & 1 & 0 & 0 & 0\\
0 & 0 & 1 & 0 & 0 & 0 & 1 & 0 & 0 & 0 & 0 & 0 & 1\\
0 & 0 & 0 & 0 & 0 & 1 & 0 & 0 & \alpha+1 & 0 & 1 & \alpha & 0
\end{bmatrix} \label{EQ:submatrix}
\end{equation}
where $\alpha = \frac{-1 \pm \sqrt{5}}{2}$ is a root of $x^2+x-1$.  {One can check that there is a unique way to extend the above $12 \times 13$ matrix to a
$12\times 34$ matrix with rank nine and the support of the Perles slack matrix, provided we scale one variable to one in each of the new columns, as allowed
by Lemma~\ref{lem:forest}. The resulting parametrized matrix is shown in
Figure~\ref{FIG:perles_slack}. Up to scaling, the two matrices corresponding to the two values of $\alpha$ are therefore the only elements in the slack variety.}


\begin{sidewaysfigure}
\centering
\vspace{12cm}
\[
\footnotesize{
\begin{tikzpicture}[
every left delimiter/.style={xshift=.25em},
every right delimiter/.style={xshift=-.25em},
inner sep=1pt,
label anchor/.style={tikz@label@post/.append style={anchor=#1}},
matrixnodes/.style={text width=4.3mm,minimum height=4.2mm,minimum width=4mm,anchor=center,
      align=center,inner sep=0pt}]
\matrix(S)[matrix of math nodes,column sep=-\pgflinewidth, row sep=-\pgflinewidth, nodes={matrixnodes},
      left delimiter={[},right delimiter={]}]{
0 &[-1ex] 0 &[-1ex] 0 &[-1ex] 1 &[1ex] \alpha\!+\!1 &[2ex] 1 &[-1ex] 0 &[1ex] 0 &[2ex] 0 &[2ex] 0 &[-1ex] 0 &[-1ex] 0 &[-1ex] 0 &[1ex] 1 &[2ex] 1 &[-1ex] \alpha &[1ex] \alpha\!+\!1 &[2ex] \alpha &[3ex] 1 &[2ex] 1 &[1ex] \alpha\!+\!1 &[2ex] 0 &[1ex] 0 &[2ex] 0 &[3ex] 0 &[3ex] 0 &[3ex] 0 &[2ex] 0 &[2ex] 0 &[2ex] 0 &[3ex] 0 &[2ex] 0 &[3ex] 0 &[2ex] 0&[2ex]\\
0 &  0 &  0 &  1 &  0 &  0 &  1 &  1 &  1 &  0 &  0 &  0 & 0 &  1\!-\!\alpha &  1 &  \alpha &  \alpha &  0 &  0 &  0 &  0 &  \alpha &  1\!-\!\alpha &  1 &  \alpha\!+\!2 &  0 &  0 &  0 &  0 &  0 &  0 &  0 &  0 & 0 \\
0 &  0 &  0 &  0 &  0 &  0 &  1 &  0 &  0 &  1 &  1 &  0 &  0 &  1 &  1 &  0 &  0 &  1 &  \alpha &  0 &  0 &  1 &  0 & 0 &  0 &  1\!-\!\alpha &  \alpha\!+\!1 &  0 &  0 &  0 &  0 &  0 &  0 &  0 \\
0 &  0 &  0 &  0 &  1 &  0 &  0 &  0 &  0 &  0 &  0 & \alpha &  1 &  0 &  0 &  1 &  1 &  0 &  0 &  1 &  0 &  0 &  1 &  1 &  0 &  0 &  0 &  1 &  1 &  1 &  0 &  0 &  0 &  0 \\
0 & 0 &  0 &  0 &  0 &  0 &  0 &  \alpha\!+\!1 &  0 &  1 &  0 &  1 &  0 &  0 &  0 &  0 &  0 &  0 &  0 &  0 &  0 &  1 &  \alpha &  0 &  0 & 1 &  0 &  \alpha\!+\!2 &  1 &  0 &  1 &  \alpha\!+\!2 &  \alpha\!+\!1 &  0 \\
1 &  0 &  0 &  0 &  0 &  0 &  0 &  0 &  0 &  0 &  1 &  0 &  \alpha & 0 &  0 &  0 &  0 &  1\!-\!\alpha &  0 &  \alpha &  0 &  0 &  0 &  0 &  0 &  0 &  1 &  0 &  0 &  \alpha\!+\!1 &  0 &  0 &  0 &  1 \\
0 & \alpha &  0 &  0 &  0 &  0 &  0 &  0 &  1 &  0 &  0 &  0 &  0 &  0 &  0 &  0 &  0 &  0 &  0 &  0 &  0 &  0 &  0 &  1\!-\!\alpha & 1 &  0 &  0 &  0 &  0 &  1 &  1\!-\!\alpha &  0 &  0 &  1 \\
0 &  0 &  1 &  0 &  0 &  1 &  0 &  0 &  0 &  0 &  0 &  0 &  0 & 0 &  0 &  0 &  0 &  0 &  1\!-\!\alpha &  0 &  1 &  0 &  0 &  0 &  0 &  \alpha &  0 &  0 &  0 &  0 &  1 &  1 &  1 &  1 \\
\alpha &  0 & 0 &  \alpha &  0 &  0 &  0 &  1 &  0 &  0 &  0 &  0 &  0 &  0 &  0 &  0 &  0 &  0 &  0 &  0 &  1\!-\!\alpha &  0 &  0 &  0 &  1 & 0 &  0 &  \alpha &  0 &  0 &  0 &  1 &  1\!-\!\alpha &  0 \\
0 &  1 &  0 &  0 &  1 &  0 &  0 &  0 &  0 &  1 &  0 &  0 &  0 &  1 & 0 &  0 &  0 &  1 &  1 &  0 &  1 &  0 &  0 &  0 &  0 &  0 &  1 &  1 &  1\!-\!\alpha &  0 &  0 &  1 &  0 &  0 \\
0 &  0 &  1 & 0 &  0 &  0 &  1 &  0 &  0 &  0 &  0 &  0 &  1 &  0 &  0 &  1 &  0 &  0 &  0 &  0 &  0 &  0 &  1 &  \alpha\!+\!1 &  \alpha\!+\!1 &  0 & \alpha &  0 &  \alpha &  0 &  0 &  0 &  0 &  0 \\
0 &  0 &  0 &  0 &  0 &  1 &  0 &  0 &  \alpha\!+\!1 &  0 &  1 &  \alpha &  0 &  0 &  1 & 0 &  1 &  0 &  0 &  1 &  0 &  1 &  0 &  0 &  0 &  1 &  0 &  0 &  0 &  \alpha\!+\!2 &  \alpha\!+\!1 &  0 &  1 &  \alpha\!+\!2\\};
{\tikzstyle{every node}=[font=\footnotesize,label anchor=center,left,xshift=-2mm]
\node at (S-1-1.west) {$A$};
\node at (S-2-1.west) {$B$};
\node at (S-3-1.west) {$C$};
\node at (S-4-1.west) {$D$};
\node at (S-5-1.west) {$E$};
\node at (S-6-1.west) {$F$};
\node at (S-7-1.west) {$G$};
\node at (S-8-1.west) {$H$};
\node at (S-9-1.west) {$-F$};
\node at (S-10-1.west) {$-G$};
\node at (S-11-1.west) {$-H$};
\node at (S-12-1.west) {$-I$};}
\end{tikzpicture}
}\]
\caption{Slack matrix of the Perles polytope}
\label{FIG:perles_slack}
\end{sidewaysfigure}


\begin{theorem}
The slack ideal of the Perles polytope is not prime.
\end{theorem}

\begin{proof}

{Let us consider the polynomial
\begin{align*}
f(\xx) & = (x_{10}x_{15}x_{36})^2 + (x_{10}x_{15}x_{36})(x_{9}x_{16}x_{35}) - (x_{9}x_{16}x_{35})^2 \\
	& = (x_{10}x_{15}x_{36} - \alpha_1x_9x_{16}x_{35})(x_{10}x_{15}x_{36} - \alpha_2x_9x_{16}x_{35}),
\end{align*}
where $\alpha_1 = \frac{-1+\sqrt{5}}{2}, \alpha_2 = \frac{-1-\sqrt{5}}{2}$ are the roots of $x^2+x-1$.
We see that the linear factor of $f(\xx)$ containing $\alpha_1$ will not vanish on the submatrix \eqref{EQ:submatrix} when we set $\alpha = \alpha_2$, and vice versa. Therefore neither of the linear factors will vanish on the slack variety. On the other hand, one can check that evaluating $f(\xx)$ on the matrix in
Figure~\ref{FIG:perles_slack} reduces it to $\alpha^2 + \alpha -1$ which is zero. Since $f(\xx)$ is homogeneous with respect to each row and column of the matrix, it will also vanish on the whole slack variety. Therefore, the vanishing ideal of the slack variety is not prime which implies that
$I_P$ is not prime.}
\end{proof}

%% file: Conclusion2.tex

\section{Conclusion} \label{sec:conclusion}

We have shown that the slack ideal of a polytope $P$ may not be prime. However, 
the following question remains.

\begin{problem}
Is $I_P$ a radical ideal? If not, what are the simplest counterexamples?
\end{problem}

{We have seen in Section~\ref{sec:Perles} that $\VV_+(I_P)$ need not be Zariski dense in the slack variety $\VV(I_P)$, and hence $I_P$ is not always the vanishing ideal of $\VV_+(I_P)$. }
In the case that $I_P$ is toric, we know that $I_P$ is  indeed the vanishing ideal of $\VV_+(I_P)$, providing a perfect correspondence between algebra and geometry.
Many further questions remain. In particular, what sort of polytope combinatorics lead to simple
algebraic structure in slack ideals?

\begin{problem}
What conditions on $P$ make its slack ideal toric, or pure difference binomial, or binomial?
\end{problem}

In fact, so far, we have not been able to find any slack ideal that is in one of these classes but not in the others.

\begin{problem}
Is any of the inclusions
\[
\textup{toric}  \subseteq \textup{pure difference binomial}  \subseteq \textup{binomial}
\]
of classes of slack ideals strict?
\end{problem}

 We also characterized the toric slack ideals that come from projectively unique polytopes as being $T_P$, the toric ideal of $G_P$, the graph of vertex-facet non-incidences of $P$. Such slack ideals were called graphic.
 The fact that testing and certifying projective uniqueness is easy for toric slack ideals is very interesting, as that is in general a hard problem. This raises the question of finding other classes of polytopes for which one can certify projective uniqueness easily.

\begin{problem}
Is there another class of polytopes, beyond those with graphic slack ideals, for which one can characterize
projective uniqueness?
\end{problem}

Apart from these concrete questions, many others could be formulated. There are, in particular, two general directions of study
that can potentially be very fruitful: strengthening the correspondence between algebraic invariants and combinatorial properties, and revisiting the literature on realization spaces in our new language
to see if further insights can be gained or open questions can be answered.